\documentclass[12pt,reqno]{amsproc}
\usepackage[hyphens]{url} 
\numberwithin{equation}{section}

\makeatletter
\@namedef{subjclassname@2020}{%
  \textup{2020} Mathematics Subject Classification}
\makeatother

\usepackage{amsmath,amssymb,color}
\usepackage{amsthm} 
\usepackage{multicol}
\usepackage{mathtools}
\usepackage{mathptmx}
\usepackage[bookmarksnumbered,colorlinks]{hyperref}
\usepackage{enumitem}
\mathtoolsset{showonlyrefs}
\usepackage[pagewise,mathlines]{lineno}

\theoremstyle{plain} 
\newtheorem{theorem}{\indent\sc Theorem}[section]
\newtheorem{lemma}[theorem]{\indent\sc Lemma}
\newtheorem{corollary}[theorem]{\indent\sc Corollary}
\newtheorem{proposition}[theorem]{\indent\sc Proposition}

\theoremstyle{definition} 

\newtheorem{remark}[theorem]{\indent\sc Remark}

\title[An integral inequality for compact Bach-flat $\mathcal{A}_{2}$-manifolds
]{\sc An integral inequality for compact Bach-flat $\mathcal{A}_{2}$-manifolds}

\bigskip

\author[F.R. dos Santos and E.J. Santos]{F\'abio R. dos Santos$^{\ast}$ and Elisa J. Santos\\
}

\address{
Departamento de Matem\'atica \\
Universidade Federal de Pernambuco \\
50.740-540 Recife, Pernambuco \\
Brazil}
\email{fabio.reis@ufpe.br}
\email{elisa.santos@ufpe.br}

\keywords{Bach-flat manifolds, Catino-type integral inequality, $\mathcal{A}_2$-manifolds, rigidity results}

\subjclass[2020]{Primary 53C20; Secondary 53C21, 53C24}
\thanks{$^{\ast}$Corresponding author}

\begin{document}

\begin{abstract}
We establish a Catino-type integral inequality for compact Bach-flat $\mathcal A_2$-manifolds with positive $\sigma_{2}$-curvature, where $\mathcal A_2$ denotes the class of Riemannian manifolds with nonnegative scalar curvature and constant nonnegative second Schouten curvature. In the equality case, we prove that the manifold is either Einstein or isometrically covered by $\mathbb{S}^1\times\mathbb{S}^{n-1}(\kappa)$, $\kappa>0$, endowed with the product metric.
\end{abstract}


\maketitle

\section{Introduction} 

In recent years, integral pinching inequalities have played an important role in the classification of Riemannian manifolds under curvature assumptions. An interesting contribution in this direction is due to
Catino~\cite{Catino:16}, who proved an optimal integral inequality for compact conformally flat manifolds with positive constant scalar curvature. More precisely, if $(M^{n},g)$ is a compact conformally flat
Riemannian manifold with positive constant scalar curvature, then 
\begin{equation}
\int_{M}|E_{g}|^{\frac{n-2}{n}}\left(R_{g}-\sqrt{n(n-1)}|E_{g}|\right)dM_{g}\leq0,
\end{equation}
where $R_{g}$ denotes the scalar curvature, $E_{g}$ the traceless Ricci tensor, and $dM_{g}$ the Riemannian volume element. Moreover, equality holds if and only if $(M^{n},g)$ is isometrically covered by either
$\mathbb{S}^{n}$ with the round metric, $\mathbb{S}^{1}\times\mathbb{S}^{n-1}$ with the product metric, or
$\mathbb{S}^{1}\times\mathbb{S}^{n-1}$ endowed with a rotationally symmetric Derdzi'nski metric.

Motivated by Catino's result, it is natural to ask whether similar integral inequalities and rigidity phenomena remain valid for larger classes of manifolds or under weaker curvature assumptions. In this
direction, Huang and Ma~\cite{Huang:16} extended Catino's integral inequality to compact manifolds with harmonic curvature tensor and positive scalar curvature, obtaining analogous rigidity conclusions.
Chu and Fang~\cite{Yawei:17} obtained rigidity results for compact manifolds with parallel Cotton tensor and positive constant scalar curvature under pointwise or integral pinching conditions. Related rigidity results under integral pinching conditions for compact manifolds with harmonic curvature and positive scalar curvature were obtained by Hai-Ping Fu~\cite{Hai:17}. More recently, Huang and Zeng~\cite{Huang:20} proved an integral inequality for closed Riemannian manifolds with positive scalar curvature, generalizing the results of Catino and Huang-Ma.

On the other hand, in his seminal work on conformal gravitation, Rudolf Bach~\cite{Bach:21} computed the first variation of the Weyl functional
\begin{equation}
\mathcal{W}(g)=\int_{M}|W_{g}|^{2}dM_{g},
\end{equation}
where $|W_{g}|$ denotes the norm of the Weyl curvature tensor. In dimension four, the Euler-Lagrange equation associated with this functional gives rise to a natural conformally invariant, symmetric,
and traceless tensor, now called the Bach tensor, whose vanishing characterizes the critical points of $\mathcal{W}$. More generally, the Bach tensor can be expressed in terms of the Weyl tensor and the
divergence of the Cotton tensor $C_g$. Denoting by $S_g$ the Schouten
tensor, it is given by
\begin{equation}\label{eq:def_bach}
B_{g}(X)=\frac{1}{n-2}\left(({\rm div},C_g)(X)+W_{g}(\cdot,X)S_{g}(\cdot)\right)
\footnote{Here, $B_g$ denotes the self-adjoint $(1,1)$-tensor associated with the Bach tensor via the metric $g$.}, \qquad n\geq3.
\end{equation}
A manifold is called Bach-flat if $B_g=0$. This class includes several important geometric structures. In particular, Einstein and conformally flat manifolds provide fundamental examples. In dimension four,
self-dual and anti-self-dual metrics, as well as locally conformally Einstein metrics, furnish further examples. The work of Derdzi'nski~\cite{Derdzinski:83} provides important nontrivial examples in this setting and illustrates the richness of the geometry beyond the Einstein and conformally flat cases.

In a recent paper~\cite{Santos:24}, the authors introduced the notion of an $\mathcal{A}_{2}$-manifold, namely a Riemannian manifold with nonnegative scalar curvature and constant nonnegative second Schouten
curvature $\sigma_{2}(g)$. In that work, they obtained characterizations of complete conformally flat manifolds by employing an Omori-Yau maximum principle and established a Catino-type integral inequality for this class of manifolds. Motivated by these developments, we extend this analysis to compact Bach-flat $\mathcal{A}_{2}$-manifolds and obtain a Catino-type integral inequality in the larger range $p\geq1$. Our main result reads as follows.
\begin{theorem}\label{teo:1.1}
Let $(M^n,g)$, $n\geq3$, be a compact Bach-flat $\mathcal{A}_2$-manifold with positive $\sigma_{2}$-curvature. Then, for every $p\geq1$, we have
\begin{equation}\label{eq:1.0}
\int_M |E_{g}|^{2p}\left(2\sqrt{|E_{g}|^2+2\sigma_{2}(g)}-(n-2)|E_{g}|\right)dM_g\leq\sqrt{\frac{2(n-2)^3}{n}}\int_M |E_{g}|^{2p}|W_g|dM_g.
\end{equation}
Moreover, equality holds if and only if $(M^n,g)$ is either Einstein or isometrically covered by
$\mathbb{S}^1\times\mathbb{S}^{n-1}(\kappa)$, for some $\kappa>0$, endowed with the product metric.
\end{theorem}

The proof of Theorem~\ref{teo:1.1} is given in Section~\ref{proof}.
\medskip

\section{Some preliminaries}\label{sec:preliminaries}

Let $(M^{n},\nabla,g)$ be a connected Riemannian manifold of dimension $n\geq3$ endowed with a Riemannian metric tensor $g$ and Levi-Civita connection $\nabla$. For simplicity, we will write only $M^{n}$ to denote the triple $(M^{n},\nabla,g)$. To fix the notation, we recall the Schouten tensor $S_{g}:\mathfrak{X}(M)\times\mathfrak{X}(M)\to\mathcal{C}^{\infty}(M)$, which is defined by
\begin{equation}\label{eq:1.1}
S_{g}={\rm Ric}_{g}-\frac{R_{g}}{2(n-1)}\,g,
\end{equation}
where ${\rm Ric}_{g}$ and $R_{g}$ denote the Ricci curvature and the scalar curvature of $M^{n}$, respectively. In what follows, we shall regard $S_{g}$ as a $(1,1)$-tensor, that is,
\begin{equation}
S_{g}(X,Y)=g(S_{g}(X),Y),\quad X,Y\in\mathfrak{X}(M).
\end{equation}
We define the squared norm of the Schouten tensor and the first Schouten curvature by
\begin{equation}
|S_{g}|={\rm tr}_{g}(S_{g}^{2})\quad\mbox{and}\quad \sigma_{1}(g):={\rm tr}_{g}(S_{g})=\frac{n-2}{2(n-1)}R_{g},
\end{equation}
where the trace ${\rm tr}_{g}$ is taking in the metric $g$ and $|\cdot|$ is the Hilbert-Schimidt norm of $S_{g}$.

In terms of the Schouten tensor, we can express the Weyl curvature tensor $\mathcal{W}_{g}$ by
\begin{equation}\label{eq:1.2}
R(X,Y)Z=\mathcal{W}_{g}(X,Y)Z+\frac{1}{n-2}(S_{g}\odot g)(X,Y)Z,
\end{equation}
where $R$ denotes the Riemann curvature $(1,3)$-tensor defined by
\begin{equation}\label{eq:riemann-curvature}
R(X,Y)Z=-\nabla_{X}\nabla_{Y}Z+\nabla_{Y}\nabla_{X}Z+\nabla_{[X,Y]}Z,
\end{equation}
and $\odot$ denotes the Kulkarni-Nomizu product:
\begin{equation}\label{eq:1.3}
\begin{split}
(S_{g}\odot g)(X,Y)Z&=g(S_{g}(X),Z)Y+g(X,Z)S_{g}(Y)\\
&\quad-g(Y,Z)S_{g}(X)-g(S_{g}(Y),Z)X,
\end{split}
\end{equation}
for all $X,Y,Z\in\mathfrak{X}(M)$. 

Related to the Schouten tensor, one can define the Cotton tensor by
\begin{equation}\label{eq:1.4}
C_g(X,Y)=\nabla S_{g}(Y,X)-\nabla S_{g}(X,Y),
\end{equation}
where $\nabla S_{g}:\mathfrak{X}(M)\times\mathfrak{X}(M)\to\mathfrak{X}(M)$ denotes the covariant differential of $S_{g}$,
\begin{equation}\label{eq:1.5}
\nabla S_{g}(X,Y)=\nabla_{Y}S_{g}(X)-S_{g}(\nabla_{Y}X),
\end{equation}
for all $X,Y\in\mathfrak{X}(M)$. Geometrically, the Cotton tensor measures the failure of the Schouten tensor to be a Codazzi tensor; that is, it vanishes if and only if $S_{g}$ is Codazzi.

As a byproduct of all this, our first result is the following general Weitzenb\"ock-type formula:
\begin{proposition}
Let $M^{n}$ be a Riemannian manifold. Then
\begin{equation}\label{eq:1.6}
\begin{split}
\frac{1}{2}\Delta_{g}|S_{g}|^{2}&=|\nabla S_{g}|^{2}+{\rm tr}(S_{g}\circ\nabla^{2}\sigma_{1}(g))+(n-2){\rm tr}(B_{g}\circ S_{g})\\
&\quad+\frac{1}{n-2}\left(n\,{\rm tr}(S_{g}^{3})-\sigma_{1}(g)|S_{g}|^{2}\right)-2\sum_{i,j=1}^{n}g(W_{g}(e_{i},e_{j})S_{g}(e_{i}),S_{g}(e_{j})),
\end{split}
\end{equation}
where $\nabla^{2}\sigma_{1}(g)$ stands for the self-adjoint linear operator metrically equivalent to the Hessian of $\sigma_{1}(g)$.
\end{proposition}

\begin{proof}
First, we recall that the divergence of the Cotton tensor is defined by
\begin{equation}\label{eq:1.7}
({\rm div}_gC_g)(X)=\sum_{i=1}^n(\nabla_{e_i}C_g)(e_i,X),
\end{equation}
where $\{e_1,\ldots,e_n\}$ is a local orthonormal frame on $M^n$. In particular, by standard tensorial properties, \eqref{eq:1.7} can be written as follows:
\begin{equation}\label{eq:1.8}
({\rm div}_gC_g)(X)=\sum_{i=1}^n\left(\nabla^2S_g(X,e_i,e_i)-\nabla^2S_g(e_i,X,e_i)\right),
\end{equation}
where $\nabla^2S_g:\mathfrak{X}(M)\times\mathfrak{X}(M)\times\mathfrak{X}(M)\to\mathfrak{X}(M)$ denotes the second covariant derivative of $S_g$, defined by $\nabla^2S_g(X,Y,Z)=\nabla_Z(\nabla S_g(X,Y))$. Moreover, we also have
\begin{equation}\label{eq:1.10}
\frac{1}{2}\Delta_g|S_g|^2=g(\Delta S_g,S_g)+|\nabla S_g|^2,
\end{equation}
where $\Delta_g$ denotes the standard Laplacian operator with respect to the metric $g$, and $\Delta\,S_g:\mathfrak{X}(M)\to\mathfrak{X}(M)$ is the rough Laplacian given by
\begin{equation}\label{eq:1.11}
\Delta\,S_g(X):={\rm tr}\!\left(\nabla^2S_g(X,\cdot,\cdot)\right)
=\sum_{i=1}^n\nabla^2S_g(X,e_i,e_i).
\end{equation}
By substituting this into the previous identity, we obtain
\begin{equation}\label{eq:1.12}
\begin{split}
({\rm div}_{g}C_g)(X)=\Delta\,S_g(X)-\sum_{i=1}^n\Big(\nabla^2S_g(e_i,e_i,X)-R(e_i,X)S_g(e_i)+S_g(R(e_i,X)e_i)\Big),
\end{split}
\end{equation}
where we have used the Ricci identity:
\begin{equation}\label{eq:1.13}
\nabla^2S_g(X,Y,Z)=\nabla^2S_g(X,Z,Y)-R(Z,Y)S_g(X)+S_g(R(Z,Y)X).
\end{equation}
Combining the above identities, and using~\eqref{eq:def_bach}, we obtain
\begin{equation}\label{eq:1.14}
\begin{split}
(n-2)B_g(X)&=\Delta S_g(X)-\sum_{i=1}^n\Big(\nabla^2S_g(e_i,e_i,X)-R(e_i,X)S_g(e_i)+S_g(R(e_i,X)e_i)\Big)\\
&\quad+\sum_{i=1}^nW_g(e_i,X)S_g(e_i).
\end{split}
\end{equation}
Then, from the previous relations, we have
\begin{equation}\label{eq:1.15}
\begin{split}
g(\Delta S_g,S_g)&=\sum_{i,j=1}^n g(\nabla^2S_g(e_i,e_i,e_j),S_g(e_j))+(n-2)\,{\rm tr}(B_g\circ S_g)\\
&\quad-\sum_{i,j=1}^n\Big(g(R(e_i,e_j)S_g(e_i),S_g(e_j))-g(S_g(R(e_i,e_j)e_i),S_g(e_j))\Big)\\
&\quad-\sum_{i,j=1}^n g(W_g(e_i,e_j)S_g(e_i),S_g(e_j)).
\end{split}
\end{equation}
Thus, by using the Weyl curvature decomposition~\eqref{eq:1.2}, it follows that
\begin{equation}\label{eq:1.16}
\begin{split}
-\sum_{i,j=1}^n &\Big(g(R(e_i,e_j)S_g(e_i),S_g(e_j))-g(S_g(R(e_i,e_j)e_i),S_g(e_j))\Big)\\
&=-\sum_{i,j=1}^n\Big(g(W_g(e_i,e_j)S_g(e_i),S_g(e_j))-g(S_g(W_g(e_i,e_j)e_i),S_g(e_j))\Big)\\
&\quad+\frac{1}{n-2}\Big(n\,{\rm tr}(S_g^3)-\sigma_1(g)|S_g|^2\Big).
\end{split}
\end{equation}
Since $W_g$ is trace-free, we have $\sum_{i,j=1}^n S_g(W_g(e_i,e_j)e_i)=0$. Hence,
\begin{equation}\label{eq:1.17}
\begin{split}
-\sum_{i,j=1}^n\Big(&g(R(e_i,e_j)S_g(e_i),S_g(e_j))-g(S_g(R(e_i,e_j)e_i),S_g(e_j))\Big)\\
&=-\sum_{i,j=1}^n g(W_g(e_i,e_j)S_g(e_i),S_g(e_j))+\frac{1}{n-2}\Big(n\,{\rm tr}(S_g^3)-\sigma_1(g)|S_g|^2\Big).
\end{split}
\end{equation}
Finally, using the fact that the trace commutes with the Levi-Civita connection, it is easy to check that
\begin{equation}\label{eq:1.18}
\sum_{i=1}^n\nabla^2 S_g(e_i,e_i,X)=\nabla_X \nabla\,{\rm tr}_g(S_g)=\nabla_X \nabla\,\sigma_1(g).
\end{equation}
Therefore, by combining the above equations, we obtain~\eqref{eq:1.6}.
\end{proof}

Next, let us consider $P_{g}$, the first Newton transformation of $S_{g}$. That is, $P_{g}:\mathfrak{X}(M)\to\mathfrak{X}(M)$ is the operator given by
\begin{equation}\label{eq:1.19}
P_g=\sigma_1(g)\,g-S_g.
\end{equation}
We define the $\sigma_{2}$-curvature by the following relation:
\begin{equation}\label{eq:1.20}
\sigma_{2}(g)=\frac{1}{2}\,{\rm tr_g}(P_g\circ S_g),
\end{equation}
and, in particular, we have the identity
\begin{equation}\label{eq:1.21}
\sigma_1(g)^2=|S_g|^2+2\sigma_{2}(g).
\end{equation}
Associated with the first Newton transformation, we define the second-order differential operator
\begin{equation}\label{eq:1.22}
\square_{g}(u)={\rm tr}(P_g\circ\nabla^2u), \qquad u\in\mathcal{C}^2(M).
\end{equation}
It is not difficult to see that $\square_{g}$ is elliptic if and only if $P_g$ is positive definite (see~\cite[Lemma 2.4]{Santos:24}). In particular, as observed in~\cite{Boris:14}, the first Newton transformation of $S_{g}$ is (up to sign) the Einstein tensor. By a standard tensor computation, we have
\begin{equation}\label{eq:1.23}
\square_{g}(u)=\mathrm{div}(P_g(\nabla u))-g(\mathrm{div}\,P_g,\nabla u),
\end{equation}
for every function $u\in\mathcal{C}^2(M)$, where
\begin{equation}\label{eq:1.24}
\mathrm{div}\,P_g=\mathrm{tr}_g(\nabla P_g).
\end{equation}
Since the Einstein tensor is divergence-free, that is, $\mathrm{div}\,P_g=0$, we obtain
\begin{equation}\label{eq:1.25}
\square_{g}(u)=\mathrm{tr}(P_g\circ\nabla^2u)=\mathrm{div}(P_g(\nabla u)),
\end{equation}
and thus the operator $\square_{g}$ can be regarded as a divergence-type operator. Moreover, for all $u,v\in\mathcal{C}^2(M)$, we have
\begin{equation}\label{eq:1.26}
\square_{g}(uv)=u\square_{g}(v)+v\square_{g}(u)+2\,g(P_g(\nabla u),\nabla v).
\end{equation}

By taking $u=\sigma_1(g)$ in~\eqref{eq:1.25}, we obtain
\begin{align}\label{eq:1.27}
\square_{g}(\sigma_1(g))&={\rm tr}(P_g\circ\nabla^2\sigma_1(g))\\
&=\sigma_1(g)\,{\rm tr}(\nabla^2\sigma_1(g))-{\rm tr}(S_g\circ\nabla^2\sigma_1(g))\nonumber\\
&=\frac{1}{2}\Delta\big(\sigma_1(g)^2\big)-|\nabla\sigma_1(g)|^2-{\rm tr}(S_g\circ\nabla^2\sigma_1(g)).
\end{align}
Thus, by combining equations~\eqref{eq:1.6},~\eqref{eq:1.21} and~\eqref{eq:1.27}, we obtain
\begin{equation}\label{eq:1.28}
\begin{split}
\square_{g}(\sigma_1(g))&=\Delta\sigma_{2}(g)+|\nabla S_g|^2-|\nabla\sigma_1(g)|^2+(n-2)\,{\rm tr}(B_g\circ S_g)\\
&\quad+\frac{1}{n-2}\Big(n\,{\rm tr}(S_g^3)-\sigma_1(g)|S_g|^2\Big)-2\sum_{i,j=1}^n g(W_g(e_i,e_j)S_g(e_i),S_g(e_j)).
\end{split}
\end{equation}

On the other hand, we consider the traceless Ricci tensor
\begin{equation}\label{eq:1.29}
E_{g}:={\rm Ric_g}-\frac{R_g}{n}g=S_g-\frac{1}{n}\sigma_1(g)\,g,
\end{equation}
where $\sigma_1(g):={\rm tr}(S_g)=\frac{n-2}{2(n-1)}R_g$. Then ${\rm tr}(E_{g})=0$, and
\begin{equation}\label{eq:1.30}
|E_{g}|^2=|S_g|^2-\frac{1}{n}\sigma_1(g)^2\geq0.
\end{equation}
Thus, $E_{g}=0$ is equivalent to the fact that $(M^n,g)$ is Einstein. By using this, we can write
\begin{equation}\label{eq:1.31}
S_g=E_{g}+\frac{1}{n}\sigma_1(g)\,g.
\end{equation}
Consequently,
\begin{equation}\label{eq:1.32}
n\,{\rm tr}(S_g^3)=n\,{\rm tr}(E_{g}^3)+3\sigma_1(g)|E_{g}|^2+\frac{\sigma_1(g)^3}{n},
\end{equation}
and
\begin{equation}\label{eq:1.33}
\sigma_1(g)|S_g|^2=\sigma_1(g)|E_{g}|^2+\frac{\sigma_1(g)^3}{n}.
\end{equation}
Thus, from~\eqref{eq:1.32} and~\eqref{eq:1.33}, we obtain
\begin{equation}\label{eq:1.34}
n\,{\rm tr}(S_g^3)-\sigma_1(g)|S_g|^2=n\,{\rm tr}(E_{g}^3)+2\sigma_1(g)|E_{g}|^2.
\end{equation}
Moreover, since $W_g$ and $B_g$ are traceless, it follows from~\eqref{eq:1.31} that
\begin{equation}
{\rm tr}_g(B_g\circ S_g)={\rm tr}_g(B_g\circ E_{g}),
\end{equation}
and
\begin{equation}\label{eq:1.35}
\sum_{i,j=1}^{n}g(W_{g}(e_{i},e_{j})S_{g}(e_{i}),S_{g}(e_j))=\sum_{i,j=1}^{n}g(W_{g}(e_{i},e_{j})E_{g}(e_{i}),E_{g}(e_j))
\end{equation}
Therefore, inserting~\eqref{eq:1.34} and~\eqref{eq:1.35} into~\eqref{eq:1.28}, we obtain
\begin{corollary}\label{prop:1.1}
Let $(M^{n},g)$ be a Riemannian manifold. Then
\begin{equation}\label{eq:1.36}
\begin{split}
\square_{g}(\sigma_1(g))&=\Delta \sigma_{2}(g)+|\nabla S_g |^{2}-|\nabla\sigma_1(g)|^2+(n-2){\rm tr}(B_{g}\circ E_{g})\\
&\quad+\dfrac{1}{n-2}\left(n{\rm tr} (E_{g} ^{3})+2\sigma_1(g)|E_{g} |^{2}\right)-2\sum_{i,j=1}^{n}g(W_{g}(e_{i},e_{j})E_{g}(e_{i}),E_{g}(e_{j})).
\end{split}
\end{equation}
\end{corollary}

\section{A fundamental result}

Following~\cite{Santos:24}, we shall henceforth work with \emph{$\mathcal{A}_{2}$-manifolds}. Recall that a Riemannian manifold $M^{n}$ is called an $\mathcal{A}_{2}$-manifold if both $\sigma_{1}(g)$ and $\sigma_{2}(g)$ are nonnegative, and $\sigma_{2}(g)$ is constant. Typical examples of $\mathcal{A}_{2}$-manifolds include Einstein manifolds with nonnegative scalar curvature. A wide class of examples can be constructed as follows: let $N^{n-1}$ be an Einstein manifold with nonnegative scalar curvature $R_{N}$, and consider $M^{n}=\Sigma^{1}\times_{\rho}N^{n-1}$ a Riemannian warped product endowed with the metric
\begin{equation}
g=dt^{2}+\rho(t)^{2}g_{N^{n-1}},
\end{equation}
where $\Sigma^{1}$ is a $1$-dimensional manifold and $\rho:\Sigma^{1}\to(0,+\infty)$ is a smooth positive function. From the standard formulas for warped products (see, e.g., \cite[Corollary 43]{ONeill:83}),
\begin{equation}\label{eq:3.1}
{\rm Ric}_{g}(\partial_{t},\partial_{t})=-(n-1)\frac{\rho''}{\rho}
\quad\text{and}\quad
{\rm Ric}_{g}(X,Y)=\frac{1}{\rho^{2}}\Bigl((n-2)(\kappa-(\rho')^{2})-\rho\rho''\Bigr)g(X,Y),
\end{equation}
where $R_{N}=(n-1)(n-2)\kappa$, 
for all $X,Y\in\mathfrak{X}(N^{n-1})$. Hence, the scalar curvature of $M^{n}$ is
\begin{equation}\label{eq:3.2}
R_{g}=\frac{n-1}{\rho^{2}}\Bigl((n-2)(\kappa-(\rho')^{2})-2\rho\rho''\Bigr).
\end{equation}
By combining~\eqref{eq:3.1} and~\eqref{eq:3.2} with~\eqref{eq:1.1}, we obtain
\begin{equation}\label{eq:3.3}
S_{g}(\partial_{t},\partial_{t})=-\frac{(n-2)}{2\rho^{2}}(2\rho\rho''+\kappa-(\rho')^{2})
\quad\text{and}\quad
S_{g}(X,Y)=\frac{(n-2)}{2\rho^{2}}(\kappa-(\rho')^{2})g(X,Y),
\end{equation}
for all $X,Y\in\mathfrak{X}(N^{n-1})$. Therefore, $S_{g}$ has one eigenvalue of multiplicity $n-1$ and another of multiplicity $1$, namely,
\begin{equation}\label{eq:3.4}
\mu=\frac{(n-2)}{2\rho^{2}}(\kappa-(\rho')^{2})
\quad\text{and}\quad
\lambda=-\frac{(n-2)}{2\rho^{2}}(2\rho\rho''+\kappa-(\rho')^{2}).
\end{equation}
Thus,
\begin{equation}\label{eq:3.5}
\sigma_{1}(g)=\sum_{i}\lambda_{i}=\lambda+(n-1)\mu=-\frac{(n-2)}{2\rho^{2}}\bigl(2\rho\rho''+(n-2)((\rho')^{2}-\kappa)\bigr),
\end{equation}
and
\begin{equation}\label{eq:3.6}
\begin{split}
\sigma_{2}(g)&=\sum_{i<j}\lambda_{i}\lambda_{j}=(n-1)\mu\lambda+\binom{n-1}{2}\mu^{2}\\
&=\frac{(n-1)(n-2)^{2}}{8\rho^{4}}((\rho')^{2}-\kappa)\bigl(4\rho\rho''+(n-4)((\rho')^{2}-\kappa)\bigr).
\end{split}
\end{equation}
Assume that $\rho$ satisfies
\begin{equation}\label{eq:3.7}
2\rho\rho''+(n-2)\bigl((\rho')^{2}-\kappa\bigr)\leq0
\end{equation}
and
\begin{equation}\label{eq:3.8}
\bigl((\rho')^{2}-\kappa\bigr)
\bigl(4\rho\rho''+(n-4)((\rho')^{2}-\kappa)\bigr)
-C\rho^{4}=0,
\end{equation}
for some $C\geq0$. It follows from~\eqref{eq:3.5} and~\eqref{eq:3.7} that $\sigma_1(g)\geq0$. Moreover, combining~\eqref{eq:3.6} with~\eqref{eq:3.8}, we obtain
\begin{equation}\label{eq:3.9}
\sigma_2(g)=\frac{(n-1)(n-2)^2}{8}C\geq0.    
\end{equation}
Thus $\sigma_2(g)$ is constant, and consequently $(M^n,g)$ is an $\mathcal A_2$-manifold. In particular, if $C>0$, then $(M^n,g)$ has positive $\sigma_2$-curvature. In this case, under~\eqref{eq:3.8},
condition~\eqref{eq:3.7} is equivalent to $|\rho'|<\sqrt{\kappa}$.

In the sequel, we establish an analytic estimate satisfied by the traceless Ricci tensor on such manifolds, which will play a key role in our study of $\mathcal{A}_2$-manifolds.
\begin{proposition}\label{prop:3.1}
Let $(M^{n},g)$ be a Bach-flat $\mathcal{A}_{2}$-manifold. Then
\begin{equation}\label{eq:3.10}
\begin{split}
\square_{g}(|E_{g}|^{2})\geq{}&2\sqrt{|E_{g}|^{2}+2\sigma_{2}(g)}\,|E_{g}|^{2}\,\varphi(|E_{g}|,|W_{g}|)\\
&\quad+\frac{2(n-1)}{n}\sigma_{1}(g)\left(|\nabla S_g|^{2}-|\nabla\sigma_{1}(g)|^{2}\right)\\
&\qquad+\frac{2(n-1)}{n}g\bigl(P_g(\nabla\sigma_{1}(g)),\nabla\sigma_{1}(g)\bigr),
\end{split}    
\end{equation}
where
\begin{equation}\label{eq:3.11}
\varphi(x,y)=-x+\frac{2}{n-2}\sqrt{x^{2}+2\sigma_{2}(g)}-\sqrt{\frac{2(n-2)}{n}}\,y.
\end{equation}
In particular,
\begin{equation}\label{eq:3.12}
\square_{g}(|E_{g}|^{2})\geq2\sqrt{|E_{g}|^{2}+2\sigma_{2}(g)}\,|E_{g}|^{2}\,\varphi(|E_{g}|,|W_{g}|).
\end{equation}
Moreover, assume that $M^n$ has positive $\sigma_{2}$-curvature. If equality holds in~\eqref{eq:3.12} at a point $q\in M^{n}$ such that $E_{g}(q)\neq0$, then $\nabla S_g(q)=0$, $E_{g}(q)$ has exactly two distinct eigenvalues, one of multiplicity $n-1$ and the other of multiplicity $1$, and $W_g(q)=0$.
\end{proposition}

In order to prove Proposition~\ref{prop:3.1}, we will need the classical algebraic lemma due to Okumura~\cite{Okumura:74}, 
with the equality case later completed independently by Xu~\cite{Xu:93} and by Alencar and do Carmo~\cite{Alencar:94}.
\begin{lemma}\label{lem:3.1}
Let $\mu_{1},\ldots,\mu_{n}$ be real numbers such that 
$\sum_{i}\mu_{i}=0$ and $\sum_{i}\mu_{i}^{2}=\beta^{2}$, where $\beta\ge0$. 
Then
\begin{equation*}
-\frac{n-2}{\sqrt{n(n-1)}}\,\beta^{3}\leq\sum_{i}\mu_{i}^{3}\leq\frac{n-2}{\sqrt{n(n-1)}}\,\beta^{3},
\end{equation*}
and equality holds if and only if at least $(n-1)$ of the numbers $\mu_{i}$ are equal.
\end{lemma}

\begin{proof}[\underline{Proof of Proposition~\ref{prop:3.1}}] Since $(M^{n},g)$ is an $\mathcal{A}_{2}$-manifold, it follows from~\eqref{eq:1.21},~\eqref{eq:1.26}, and~\eqref{eq:1.30} that
\begin{equation}\label{eq:3.13}
\frac{1}{2}\square_{g}(|E_{g}|^2)=\frac{n-1}{n}\sigma_1(g)\square_{g}(\sigma_1(g))+\frac{n-1}{n}
g\left(P_g(\nabla\sigma_1(g)),\nabla\sigma_1(g)\right).
\end{equation}
Combining~\eqref{eq:3.13} with Corollary~\ref{prop:1.1}, and using $B_g=0$ and the constancy of $\sigma_2(g)$, we obtain
\begin{equation}\label{eq:3.14}
\begin{split}
\frac{n}{2(n-1)}\square_{g}(|E_{g}|^2)&=\sigma_1(g)\left(|\nabla S_g|^2-|\nabla\sigma_1(g)|^2\right)+g\left(P_g(\nabla\sigma_1(g)),\nabla\sigma_1(g)\right)\\
&\quad+\frac{\sigma_1(g)}{n-2}\left(n\operatorname{tr}(E_g^3)+2\sigma_1(g)|E_g|^2\right)-2\sigma_1(g)\sum_{i,l=1}^{n}W_g(e_i,e_l,E_g(e_i),E_g(e_l))\\
\end{split}
\end{equation}
Since $n\geq3$, Lemma~\ref{lem:3.1} gives
\begin{equation}\label{eq:3.15}
\left|\operatorname{tr}(E_g^3)\right|\leq\frac{n-2}{\sqrt{n(n-1)}}|E_g|^3,
\end{equation}
and therefore
\begin{equation}\label{eq:3.16}
n\sigma_1(g)\operatorname{tr}(E_g^3)\geq-n\sigma_1(g)\left|\operatorname{tr}(E_g^3)\right|\geq
-\frac{n(n-2)}{\sqrt{n(n-1)}}\sigma_1(g)|E_g|^3.
\end{equation}
Inserting this estimate into~\eqref{eq:3.14}, we obtain
\begin{equation}\label{eq:3.17}
\begin{split}
\frac{n}{2(n-1)}\square_{g}(|E_{g}|^2)&\geq\frac{\sigma_1(g)}{n-2}\left(-\frac{n(n-2)}{\sqrt{n(n-1)}}|E_g|^3+2\sigma_1(g)|E_g|^2\right)\\
&\quad+\sigma_1(g)\left(|\nabla S_g|^2-|\nabla\sigma_1(g)|^2\right)+g\left(P_g(\nabla\sigma_1(g)),\nabla\sigma_1(g)\right)\\
&\qquad-2\sigma_1(g)\sum_{i,l=1}^{n}W_g(e_i,e_l,E_g(e_i),E_g(e_l)).
\end{split}
\end{equation}
Moreover, from~\eqref{eq:1.21} and~\eqref{eq:1.30},
\begin{equation}\label{eq:3.18}
\sigma_1(g)^2=\frac{n}{n-1}\left(|E_g|^2+2\sigma_2(g)\right)\quad\mbox{and}\quad\sigma_1(g)=\sqrt{\frac{n}{n-1}}
\sqrt{|E_g|^2+2\sigma_2(g)}.
\end{equation}
Substituting these identities into~\eqref{eq:3.17} yields
\begin{equation}\label{eq:3.19}
\begin{split}
\square_g(|E_g|^2)&\geq\frac{2}{n-2}\sqrt{|E_g|^2+2\sigma_2(g)}\left(-(n-2)|E_g|
+2\sqrt{|E_g|^2+2\sigma_2(g)}\right)|E_g|^2\\
&\quad+\frac{2(n-1)}{n}\left(\sigma_1(g)\left(|\nabla S_g|^2-|\nabla\sigma_1(g)|^2\right)+g\left(P_g(\nabla\sigma_1(g)),\nabla\sigma_1(g)\right)\right)\\
&\quad\quad-\frac{4(n-1)}{n}\sqrt{\frac{n}{n-1}}\sqrt{|E_g|^2+2\sigma_2(g)}\sum_{i,l=1}^{n}
W_g(e_i,e_l,E_g(e_i),E_g(e_l)).
\end{split}
\end{equation}
For the last term,~\cite[Lemma~3.4]{Huisken:85} gives
\begin{equation}\label{eq:3.20}
\begin{split}
-\sum_{i,l=1}^{n}W_g(e_i,e_l,E_g(e_i),E_g(e_l))&\geq-\sum_{i,l=1}^{n}\left|W_g(e_i,e_l,E_g(e_i),E_g(e_l))
\right|\\
&\geq-\sqrt{\frac{n-2}{2(n-1)}}\,|E_g|^2|W_g|.
\end{split}
\end{equation}
Inserting~\eqref{eq:3.20} into~\eqref{eq:3.19}, we obtain~\eqref{eq:3.10}.

Moreover, by~\cite[Lemma~2.4]{Santos:24},
\begin{equation}\label{eq:3.21}
|\nabla S_g|^2-|\nabla\sigma_1(g)|^2\geq0,
\end{equation}
and by~\cite[Lemma~3.4]{Santos:24}, $P_g$ is positive semi-definite. Hence
\begin{equation}\label{eq:3.22}
g\left(P_g(\nabla\sigma_1(g)),\nabla\sigma_1(g)\right)\geq0.
\end{equation}
Therefore,~\eqref{eq:3.10} implies~\eqref{eq:3.12}.

We finally analyze the equality case. Let $q\in M^n$ be a point at which equality holds in~\eqref{eq:3.12}, and assume that $E_g(q)\neq0$. Since $\sigma_2(g)>0$, \cite[Lemma~3.4]{Santos:24} implies that $P_g$ is positive definite. Moreover, $\sigma_1(g)>0$. Since~\eqref{eq:3.12} is obtained from~\eqref{eq:3.10} by discarding two nonnegative terms, equality in~\eqref{eq:3.12} implies
\begin{equation}\label{eq:3.23}
g\left(P_g(\nabla\sigma_1(g)),\nabla\sigma_1(g)\right)(q)=0\quad\mbox{and}\quad|\nabla S_g|^2(q)-|\nabla\sigma_1(g)|^2(q)=0.
\end{equation}
The positive definiteness of $P_g$ gives $\nabla\sigma_1(g)(q)=0$, and consequently, $\nabla S_g(q)=0$.

Tracing back the inequalities leading to~\eqref{eq:3.12}, equality must also hold in the Okumura estimate~\eqref{eq:3.15}. Since $E_g(q)\neq0$, Lemma~\ref{lem:3.1} then implies that $E_g(q)$
has exactly two distinct eigenvalues, one of multiplicity $n-1$ and the other of multiplicity $1$. Thus,
$T_qM=M_\lambda\oplus M_\mu$, with $\dim M_\lambda=n-1$ and $\dim M_\mu=1$, where $E_g|_{M_\lambda}=\lambda g$ and $E_g(v,v)=\mu$ for any unit vector $v\in M_\mu$. Hence, at $q$,
\begin{equation}\label{eq:3.24}
E_g=\lambda g+(\mu-\lambda)v^\ast\otimes v^\ast,
\end{equation}
where $v^\ast$ denotes the $1$-form dual to $v$.

Using~\eqref{eq:3.24}, we compute
\begin{equation}
\begin{split}
\sum_{i,l}W_g(e_i,e_l,E_g(e_i),E_g(e_l))&=\lambda^2\sum_{i,j,k,l}W_g(e_i,e_k,e_j,e_l)g(e_i,e_j)g(e_k,e_l)\\
&\quad+2\lambda(\mu-\lambda)\sum_{i,j,k,l}W_g(e_i,e_k,e_j,e_l)g(e_i,e_j)v^\ast(e_k)v^\ast(e_l)\\
&\qquad+(\mu-\lambda)^2\sum_{i,j,k,l}W_g(e_i,e_k,e_j,e_l)v^\ast(e_i)v^\ast(e_j)v^\ast(e_k)v^\ast(e_l).
\end{split}
\end{equation}
By the symmetries of the Weyl tensor, the Bianchi identity, and its trace-free property,
\begin{equation}
\sum_{i,l}W_g(e_i,e_l,E_g(e_i),E_g(e_l))(q)=0.
\end{equation}
Equality must also hold in~\eqref{eq:3.20}; hence
\begin{equation}
0=\sqrt{\frac{n-2}{2(n-1)}}\,|W_g(q)|\,|E_g(q)|^2.
\end{equation}
Since $E_g(q)\neq0$, we conclude that $W_g(q)=0$. This completes the equality analysis.
\end{proof}

\section{Proof of Theorem~\ref{teo:1.1}}\label{proof}
\medskip

For simplicity, set $u=|E_g|^2$. Since $\sigma_2(g)>0$, the identity~\eqref{eq:3.18} implies that $\sigma_1(g)>0$. Moreover, $\sigma_2(g)$ is constant, and differentiating~\eqref{eq:3.18} gives
\begin{equation}\label{eq:4.1}
\nabla\sigma_1(g)=\frac{n}{2(n-1)\sigma_1(g)}\,\nabla u.
\end{equation}
Consequently,
\begin{equation}\label{eq:4.2}
\frac{2(n-1)}{n}g\bigl(P_g(\nabla\sigma_1(g)),\nabla\sigma_1(g)\bigr)=\frac{1}{2\bigl(u+2\sigma_2(g)\bigr)}
g\bigl(P_g(\nabla u),\nabla u\bigr).
\end{equation}
Therefore, the estimate~\eqref{eq:3.10} in Proposition~\ref{prop:3.1} yields
\begin{equation}\label{eq:4.3}
\begin{split}
\square_g(u)&\geq2u\sqrt{u+2\sigma_2(g)}\,\varphi(\sqrt{u},|W_g|)+\frac{2(n-1)}{n}\sigma_1(g)\left(|\nabla S_g|^2-|\nabla\sigma_1(g)|^2\right)\\
&\qquad+\frac{1}{2\bigl(u+2\sigma_2(g)\bigr)}g\bigl(P_g(\nabla u),\nabla u\bigr),
\end{split}
\end{equation}
where $\varphi$ is given in~\eqref{eq:3.11}.

Fix $p\geq1$. For $\varepsilon>0$, consider
\begin{equation}\label{eq:4.4}
f_\varepsilon(t)=\frac{(t+\varepsilon)^{p-1}}{\sqrt{t+2\sigma_2(g)}},\qquad t\geq0.
\end{equation}
The compactness of $M^n$ implies that $u(M)\subset[0,A]$ for some $A>0$. So, for every $\varepsilon>0$, the function $f_\varepsilon$ extends smoothly to an open interval containing $[0,A]$. By~\eqref{eq:1.25},
\begin{equation}\label{eq:4.5}
f_\varepsilon(u)\square_g(u)=\operatorname{div}\bigl(f_\varepsilon(u)P_g(\nabla u)\bigr)-f_\varepsilon'(u)
g\bigl(P_g(\nabla u),\nabla u\bigr).
\end{equation}
Integrating over $M^n$ and applying Stokes' theorem, we obtain
\begin{equation}\label{eq:4.6}
\int_M f_\varepsilon(u)\square_g(u)\,dM_g=-\int_Mf_\varepsilon'(u)g\bigl(P_g(\nabla u),\nabla u\bigr)\,dM_g.
\end{equation}
Multiplying~\eqref{eq:4.3} by $f_\varepsilon(u)\geq0$, we have
\begin{equation}\label{eq:4.7}
\begin{split}
f_\varepsilon(u)\square_g(u)&\geq2u(u+\varepsilon)^{p-1}\varphi(\sqrt{u},|W_g|)+\frac{2(n-1)}{n}\sigma_1(g)f_\varepsilon(u)\left(|\nabla S_g|^2-|\nabla\sigma_1(g)|^2\right)\\
&\quad+\frac{f_\varepsilon(u)}{2\bigl(u+2\sigma_2(g)\bigr)}g\bigl(P_g(\nabla u),\nabla u\bigr).
\end{split}
\end{equation}
Combining~\eqref{eq:4.6} and~\eqref{eq:4.7}, we obtain
\begin{equation}\label{eq:4.8}
\begin{split}
0&\geq2\int_Mu(u+\varepsilon)^{p-1}\varphi(\sqrt{u},|W_g|)\,dM_g\\
&\quad+\frac{2(n-1)}{n}\int_M\sigma_1(g)f_\varepsilon(u)\left(|\nabla S_g|^2-|\nabla\sigma_1(g)|^2\right)dM_g\\
&\qquad+\int_M\left(f_\varepsilon'(u)+\frac{f_\varepsilon(u)}{2\bigl(u+2\sigma_2(g)\bigr)}\right)g\bigl(P_g(\nabla u),\nabla u\bigr)\,dM_g.
\end{split}
\end{equation}
A direct computation gives
\begin{equation}\label{eq:4.9}
f_\varepsilon'(t)=\frac{(t+\varepsilon)^{p-2}}{2(t+2\sigma_2(g))^{3/2}}\left((2p-3)t+4(p-1)\sigma_2(g)
-\varepsilon\right).
\end{equation}
Consequently,
\begin{equation}\label{eq:4.10}
f_\varepsilon'(t)+\frac{f_\varepsilon(t)}{2(t+2\sigma_2(g))}=(p-1)\frac{(t+\varepsilon)^{p-2}}{\sqrt{t+2\sigma_2(g)}},
\end{equation}
and thus, for $p\geq1$,
\begin{equation}\label{eq:4.11}
f_\varepsilon'(t)+\frac{f_\varepsilon(t)}{2(t+2\sigma_2(g))}\geq0.
\end{equation}
Furthermore, by~\cite[Lemma~2.4]{Santos:24},
\begin{equation}\label{eq:4.12}
\sigma_1(g)^2\left(|\nabla S_g|^2-|\nabla\sigma_1(g)|^2\right)\geq2\sigma_2(g)|\nabla S_g|^2.
\end{equation}
Since $\sigma_1(g)>0$ and $\sigma_2(g)>0$, it follows that
\begin{equation}\label{eq:4.13}
|\nabla S_g|^2-|\nabla\sigma_1(g)|^2\geq0.
\end{equation}
Moreover, the positivity of the $\sigma_2$-curvature implies that $P_g$ is positive definite. Hence all the last two terms in~\eqref{eq:4.8} are nonnegative, and therefore
\begin{equation}\label{eq:4.14}
0\geq\int_Mu(u+\varepsilon)^{p-1}\varphi(\sqrt{u},|W_g|)\,dM_g.
\end{equation}
By using again the compactness of $M^n$, we have that both $u$ and $|W_g|$ are bounded, and hence $\varphi(\sqrt{u},|W_g|)$ is bounded on $M^n$. Moreover, for $0<\varepsilon\leq1$,
\begin{equation}\label{eq:4.15}
0\leq u(u+\varepsilon)^{p-1}\leq u(u+1)^{p-1}.
\end{equation}
Thus, by the dominated convergence theorem, we may let $\varepsilon\to0^+$ in~\eqref{eq:4.14} to obtain
\begin{equation}\label{eq:4.16}
0\geq\int_Mu^p\varphi(\sqrt{u},|W_g|)\,dM_g.
\end{equation}
Recalling that $u=|E_g|^2$ and using the definition of $\varphi$, we arrive at
\begin{equation}\label{eq:4.17}
\int_M |E_g|^{2p}\left(2\sqrt{|E_g|^2+2\sigma_2(g)}-(n-2)|E_g|\right)dM_g\leq\sqrt{\frac{2(n-2)^3}{n}}\int_M |E_g|^{2p}|W_g|\,dM_g.
\end{equation}
This proves the inequality for every $p\geq1$.

We now analyze the equality case. Assume that equality holds in~\eqref{eq:4.17}, and set
\begin{equation}\label{eq:4.18}
I_\varepsilon=\int_Mu(u+\varepsilon)^{p-1}\varphi(\sqrt{u},|W_g|)\,dM_g.
\end{equation}
We also define
\begin{equation}\label{eq:4.19}
J_\varepsilon=\frac{2(n-1)}{n}\int_M\sigma_1(g)f_\varepsilon(u)\left(|\nabla S_g|^2-|\nabla\sigma_1(g)|^2
\right)dM_g,
\end{equation}
and
\begin{equation}\label{eq:4.20}
L_\varepsilon=(p-1)\int_M\frac{(u+\varepsilon)^{p-2}}{\sqrt{u+2\sigma_2(g)}}g\bigl(P_g(\nabla u),\nabla u\bigr)\,dM_g.
\end{equation}
By~\eqref{eq:4.8} and~\eqref{eq:4.10},
\begin{equation}\label{eq:4.21}
0\geq 2I_\varepsilon+J_\varepsilon+L_\varepsilon.
\end{equation}
Moreover, $J_\varepsilon\geq0$ and $L_\varepsilon\geq0$. Equality in~\eqref{eq:4.17} is equivalent to
\begin{equation}\label{eq:4.22}
\int_Mu^p\varphi(\sqrt{u},|W_g|)\,dM_g=0.
\end{equation}
The dominated convergence argument above gives
\begin{equation}\label{eq:4.23}
\lim_{\varepsilon\to0^{+}}I_\varepsilon=\lim_{\varepsilon\to0^{+}}\int_Mu(u+\varepsilon)^{p-1}\varphi(\sqrt{u},|W_g|)\,dM_g=0
\end{equation}
It follows from~\eqref{eq:4.21} that $\lim_{\varepsilon\to0^{+}}(J_\varepsilon+L_\varepsilon)=0$, in particular,
\begin{equation}\label{eq:4.24}
\lim_{\varepsilon\to0^{+}}J_\varepsilon=0.
\end{equation}
We claim that
\begin{equation}\label{eq:4.25}
\nabla u=0\quad\text{on}\quad\{u>0\}.
\end{equation}
Fix $\delta>0$ and let
\begin{equation}\label{eq:4.26}
K_\delta=\{x\in M^n:u(x)\geq\delta\}.
\end{equation}
If $K_\delta=\varnothing$, there is nothing to prove. Otherwise, using~\eqref{eq:3.18} and~\eqref{eq:4.4}, we have
\begin{equation}\label{eq:4.27}
\sigma_1(g)f_\varepsilon(u)=\sqrt{\frac{n}{n-1}}\,(u+\varepsilon)^{p-1}.
\end{equation}
Therefore, on $K_\delta$,
\begin{equation}\label{eq:4.28}
\sigma_1(g)f_\varepsilon(u)\geq\sqrt{\frac{n}{n-1}}\,\delta^{p-1},
\end{equation}
and hence
\begin{equation}\label{eq:4.29}
J_\varepsilon\geq C(n,p,\delta)\int_{K_\delta}\left(|\nabla S_g|^2-|\nabla\sigma_1(g)|^2\right)dM_g,
\end{equation}
where $C(n,p,\delta)>0$ is independent of $\varepsilon$. Passing to the limit and using~\eqref{eq:4.24}, we obtain
\begin{equation}\label{eq:4.30}
\int_{K_\delta}\left(|\nabla S_g|^2-|\nabla\sigma_1(g)|^2\right)dM_g=0,
\end{equation}
and consequently
\begin{equation}\label{eq:4.31}
|\nabla S_g|^2-|\nabla\sigma_1(g)|^2=0\quad\text{on}\quad K_\delta.
\end{equation}
By~\eqref{eq:4.12} and $\sigma_2(g)>0$, the equality case in~\cite[Lemma~2.4]{Santos:24} implies that $\nabla S_g=0$ on $K_\delta$. So, from~\eqref{eq:4.31} we have $\nabla\sigma_1(g)=0$ on $K_\delta$, and~\eqref{eq:4.1} gives $\nabla u=0$ on $K_\delta$. Since $\delta>0$ is arbitrary,~\eqref{eq:4.25} follows.

We now show that $u$ is constant on $M^n$. If $\{u>0\}=\varnothing$, then $u\equiv0$. Otherwise, let $\Omega$ be a connected component of $\{u>0\}$. By~\eqref{eq:4.25}, $u$ is constant on $\Omega$,
say $u\equiv c>0$. If $x\in\partial\Omega$, continuity gives $u(x)=c>0$, so $x\in\{u>0\}$. Since $\{u>0\}$ is open, $x$ belongs to the same connected component $\Omega$, a contradiction. Thus $\partial\Omega=\varnothing$. Since $M^n$ is connected, $\Omega=M^n$. Therefore,
\begin{equation}\label{eq:4.32}
u=|E_g|^2\quad\text{is constant on}\quad M^n.
\end{equation}
If $u\equiv0$, then $E_g\equiv0$, and $(M^n,g)$ is Einstein. Otherwise, $u\equiv c>0$. In this case, $\square_g(u)=0$, and the estimate~\eqref{eq:3.12} in Proposition~\ref{prop:3.1} gives
\begin{equation}\label{eq:4.33}
\varphi(\sqrt{u},|W_g|)\leq0\quad\text{on}\quad M^n.
\end{equation}
Since $u^p$ is a positive constant,~\eqref{eq:4.22} implies
\begin{equation}\label{eq:4.34}
\varphi(\sqrt{u},|W_g|)\equiv0.
\end{equation}
Hence equality holds in the estimate~\eqref{eq:3.12} of Proposition~\ref{prop:3.1} at every point of $M^n$. Since $u=|E_g|^2>0$, we have $E_g\neq0$ everywhere. The equality case of Proposition~\ref{prop:3.1} therefore yields 
\begin{equation}\label{eq:4.35}
\nabla S_g=0\quad\mbox{and}\quad W_g=0,
\end{equation}
and $E_g$ has exactly two distinct eigenvalues with multiplicities $1$ and $n-1$.

In particular, the Ricci tensor is parallel, and $(M^n,g)$ is conformally flat. By~\cite[Corollary~1]{Boubel:01} (see also~\cite[Theorem~1.100]{Besse:87}), $(M^n,g)$ is isometrically covered by a Riemannian product of Einstein manifolds. The multiplicities $1$ and $n-1$ of the two distinct Ricci eigenvalues show
that the corresponding factors have dimensions $1$ and $n-1$, respectively. The conformal flatness of $g$ then implies that the $(n-1)$-dimensional factor has constant sectional curvature, which is positive by the positivity of the scalar curvature. Hence $(M^n,g)$ is isometrically covered by $\mathbb S^1\times\mathbb S^{n-1}(\kappa)$, for some $\kappa>0$, endowed with the product metric.

Conversely, if $(M^n,g)$ is Einstein, then $E_g\equiv0$, and both sides of~\eqref{eq:4.17} vanish. Suppose now that $(M^n,g)$ is isometrically covered by $\mathbb S^1\times\mathbb S^{n-1}(\kappa)$, for some $\kappa>0$, endowed with the product metric. By~\eqref{eq:3.4} and~\eqref{eq:3.6}, the Schouten tensor has eigenvalues $-a$ and $a$, with multiplicities $1$ and $n-1$, respectively, where $a=\frac{n-2}{2}\kappa$. Hence
\begin{equation}\label{eq:4.36}
\sigma_2(g)=\frac{(n-1)(n-4)}{2}a^2,\qquad|E_g|^2=\frac{4(n-1)}{n}a^2,
\end{equation}
and therefore
\begin{equation}\label{eq:4.37}
2\sqrt{|E_g|^2+2\sigma_2(g)}=(n-2)|E_g|.
\end{equation}
Moreover, $W_g=0$, and hence both sides of~\eqref{eq:4.17} vanish. Thus equality holds for every $p\geq1$. The preceding expression for $\sigma_2(g)$ shows that the assumption $\sigma_2(g)>0$ forces $n\geq5$ in the non-Einstein product case. Consequently, when $n=3$ or $n=4$, the equality case of~\eqref{eq:4.17} reduces to the Einstein case. This completes the proof.
\qed

\begin{remark}
The exponent $p=1$ is the natural threshold for the present argument, since the coefficient arising from the gradient term is proportional to $p-1$. It remains an open question whether the inequality in Theorem~\ref{teo:1.1} continues to hold for some $p<1$.
\end{remark}

\section*{Acknowledgements}

The first author is partially supported by CNPq, Brazil, grant 303311/2025-8, and Propesqi (UFPE). The second author is partially supported by CAPES, Brazil.

\section*{Conflict of Interest}

On behalf of all authors, the corresponding author states that there is no conflict of interest.

\section*{Data Availability Statement}

Data sharing is not applicable to this article as no datasets were generated or analyzed during the current study.

\end{document}